\documentclass{amsart}
\allowdisplaybreaks
\usepackage[parfill]{parskip}
\usepackage[noadjust]{cite}
\usepackage{epigraph}
\usepackage[colorlinks=true,
            linkcolor=red,
            urlcolor=blue,
            citecolor=magenta]{hyperref}
\usepackage{color}
\usepackage{mathrsfs}
\theoremstyle{plain}
\newtheorem{thm}{Theorem}[section]
\newtheorem*{thm*}{Theorem}
\newtheorem{prop}{Proposition}[section]
\newtheorem*{prop*}{Proposition}
\newtheorem{cor}{Corollary}[section]
\newtheorem*{cor*}{Corollary}

\newtheorem*{lem*}{Lemma}
\theoremstyle{definition}

\newtheorem*{defn*}{Definition}
\newtheorem{exmp}{Example}[section]
\newtheorem*{exmp*}{Example}

\newtheorem*{exmps*}{Examples}

\newtheorem{rem}{Remark}[section]
\newtheorem*{rem*}{Remark}

\newtheorem*{rems*}{Remarks}

\newtheorem*{note*}{Note}
\newcommand{\N}{{\mathbb N}}
\newcommand{\C}{{\mathbb C}}

\newcommand\restr[2]{\ensuremath{#1\big|{#2}}}

\begin{document}
\title[On spectral features inherent to scalar type spectral operators]
{On certain spectral features\\ inherent to scalar type spectral operators}
\author[Marat V. Markin]{Marat V. Markin}
\address{
Department of Mathematics\newline
California State University, Fresno\newline
5245 N. Backer Avenue, M/S PB 108\newline
Fresno, CA 93740-8001
}
\email{mmarkin@csufresno.edu}
\begin{abstract}
Important spectral features, such as the emptiness of the residual spectrum, countability of the point spectrum, provided the space is separable, and a characterization of spectral gap at $0$, known to hold for bounded scalar type spectral operators, are shown to naturally transfer to the unbounded case. 
\end{abstract}
\maketitle
\epigraph{\textit{Curiosity is the lust of the mind.}}{Thomas Hobbes}

\section[Introduction]{Introduction}

As is known {\cite[Theorem 8]{Dunford1954}} (see also \cite{Dunford1958,Dun-SchIII}), a \textit{bounded} linear operator $T$ on a complex Banach space $(X,\|\cdot\|)$ is \textit{spectral} iff it allows the unique \textit{canonical decomposition}
\begin{equation*}
T=S+N,
\end{equation*}
where $S$ is a \textit{scalar type spectral operator} and $N$ is a \textit{quasinilpotent operator} commuting with $S$. The operators $T$ and $S$ have the same \textit{spectrum} and \textit{spectral measure} $E(\cdot)$, with
\begin{equation}\label{SP}
S=\int\limits_{\sigma(T)} \lambda\,dE(\lambda),
\end{equation}
where $\sigma(\cdot)$ is the \textit{spectrum} of an operator, and are called the \textit{scalar} and \textit{radical parts} of $T$, respectively. 

The operator $N$ being \textit{nilpotent}, $T$ is called of \textit{finite type} (cf. \cite{Dunford1954,Foguel1958}), in which case, in particular, for a bounded scalar type spectral operator ($T=S$, $N=0$), the \textit{residual spectrum} is \textit{empty} {\cite[Theorem 4.1]{Foguel1958}} and, provided the space $X$ is \textit{separable}, the \textit{point spectrum} is \textit{countable} {\cite[Theorem 4.4]{Foguel1958}}.

Furthermore, {\cite[Theorem 3.4]{Foguel1958}} describing the closedness of the range of a bounded spectral operator $T$ on a complex Banach space $X$, when applied to a bounded scalar type spectral operator ($T=S$, $N=0$), turns into a characterization of \textit{spectral gap} at $0$ acquiring the following form:

\begin{thm}[{\cite[Theorem 3.4]{Foguel1958}}, the scalar type case]\label{SGB}\ \\
For a bounded scalar type spectral operator $A$ on a complex Banach space $(X,\|\cdot\|)$ with $0\in \sigma(A)$, $0$ is an isolated point of the spectrum $\sigma(A)$ iff the range of $A$ is closed.
\end{thm}

The case of an \textit{unbounded} spectral operator $T$ in a complex Banach space $(X,\|\cdot\|)$ appears to be essentially more formidable. Thus, $E(\cdot)$ being the spectral measure of $T$, the scalar part $S$ of $T$ defined by \eqref{SP} is an unbounded scalar type spectral operator and the radical part $N:=T-S$ need not be bounded, let alone quasinilpotent \cite{Bade1954,Dun-SchIII}.

A natural question is whether the discussed spectral features pass to unbounded spectral operators at least when they are of scalar type, which would include the important class \textit{normal operators}. In this note, we are to show that the emptiness of the residual spectrum, the countability of the point spectrum, provided the space is separable, as well as the characterization of spectral gap at $0$ are inherent to scalar type spectral operators, bounded or not.

\section{Preliminaries}\label{prelim}

Recall that, the spectrum $\sigma(A)$ of a \textit{closed linear operator} $A$ in a complex Banach space $(X,\|\cdot\|)$ is partitioned into disjoint components, $\sigma_p(A)$, $\sigma_c(A)$, and $\sigma_r(A)$, called the \textit{point}, \textit{continuous}, and \textit{residual spectrum} of $A$, respectively, as follows:
\begin{equation*}
\begin{split}
& \sigma_p(A)=\left\{\lambda\in \C \,\middle|\,A-\lambda I\ \text{is \textit{not one-to-one}, i.e., $\lambda$ is an \textit{eigenvalue} of $A$} \right\},\\
& \sigma_c(A)=\left\{\lambda\in \C \,\middle|\,A-\lambda I\ \text{is \textit{one-to-one} and $R(A-\lambda I)\neq X$, but $\overline{R(A-\lambda I)}=X$} \right\},\\
& \sigma_r(A)=\left\{\lambda\in \C \,\middle|\,A-\lambda I\ \text{is \textit{one-to-one} and $\overline{R(A-\lambda I)}\neq X$} \right\},
\end{split}
\end{equation*}
where $I$ stands for the \textit{identity} operator on $X$, $R(\cdot)$ is the \textit{range} of an operator, and $\bar{\cdot}$ is the \textit{closure} of a set in $\C$ (see, e.g., \cite{Foguel1958}).

The properties of spectral operators, spectral measures, and the Borel operational calculus underlying the subsequent discourse are exhaustively delineated in \cite{Dunford1958,Dun-SchIII}. Here, for the reader's convenience, we give an outline of some particularly important facts.

Recall that a \textit{spectral operator} is a densely defined closed linear operator $A$ in a complex Banach space $(X,\|\cdot\|)$ with an associated {\it spectral measure} ({\it resolution of the identity}) $E_A(\cdot)$, i.e., a \textit{strongly $\sigma$-additive} operator function, which assigns to each set $\delta$ from the $\sigma$-algebra $\mathscr{B}$ of {\it Borel sets}
in $\C$ a {\it projection operator} $E_A(\delta)=E_A^2(\delta)$ on $X$ and has the following properties:
\begin{equation*}
E_A(\emptyset)=0,\ E_A(\C)=I,\ E_A(\delta\cap\sigma)=E_A(\delta)E_A(\sigma)=E_A(\sigma)E_A(\delta),\ \delta,\sigma\in \mathscr{B},
\end{equation*}
where $0$ stands for the \textit{zero} operator on $X$, and
\begin{equation*}
\begin{split}
&E_A(\delta)X\subseteq D(A),\ \text{for each \textit{bounded}}\ \delta\in \mathscr{B},\\
&E_A(\delta)D(A)\subseteq D(A),\ 
AE_A(\delta)f=E_A(\delta)Af,\ \delta\in \mathscr{B},f\in D(A),\\
&\sigma(\restr{A}{E_A(\delta)X})\subseteq \bar{\delta},\ \delta\in \mathscr{B},
\end{split}
\end{equation*}
where $D(\cdot)$ is the \textit{domain} of an operator and $\restr{\cdot}{\cdot}$ is the \textit{restriction} of an operator (left) to a subspace (right).

Due to its {\it strong countable additivity}, the spectral measure $E_A(\cdot)$ is {\it bounded} \cite{Dun-SchI,Dun-SchIII}, i.e.,
\begin{equation*}
\exists\, M>0\ \forall\,\delta\in \mathscr{B}:\ \|E_A(\delta)\|\le M.
\end{equation*}
The notation $\|\cdot\|$ has been recycled here to designate the norm in the space $\mathscr{L}(X)$ of all bounded linear operators on $X$, such an economy of symbols being rather conventional. 

A spectral operator $A$ in a complex Banach space $(X,\|\cdot\|)$ with spectral measure $E_A(\cdot)$ is said to be of \textit{scalar type} if
\begin{equation*}
A=\int\limits_{\C} \lambda\,dE_A(\lambda),
\end{equation*}
which is imbedded into the structure of the \textit{Borel operational calculus}
associated with such operators  \cite{Dunford1958,Dun-SchIII} and assigning to any Borel measurable function $F:\C\to \C$ a scalar type spectral operator
\begin{equation*}
F(A):=\int\limits_\C F(\lambda)\,dE_A(\lambda)
\end{equation*}
defined as follows:
\begin{equation*}
\begin{split}
F(A)f&:=\lim_{n\to\infty}F_n(A)f,\ f\in D(F(A)),\\
D(F(A))&:=\left\{f\in X\middle| \lim_{n\to\infty}F_n(A)f\ \text{exists}\right\},
\end{split}
\end{equation*}
where
\begin{equation*}
F_n(\cdot):=F(\cdot)\chi_{\{\lambda\in\C\,|\,|F(\lambda)|\le n\}}(\cdot),
\ n\in\N,
\end{equation*}
($\chi_\delta(\cdot)$ is the {\it characteristic function} of a set $\delta\subseteq \C$, $\N:=\left\{1,2,3,\dots\right\}$ is the set of \textit{natural numbers}) and
\begin{equation*}
F_n(A):=\int\limits_{\C} F_n(\lambda)\,dE_A(\lambda),\ n\in\N,
\end{equation*}
are {\it bounded} scalar type spectral operators on $X$ defined in the same manner as for a {\it normal operator} (see, e.g., \cite{Dun-SchII,Plesner}).

The spectrum $\sigma(A)$ of a scalar type spectral operator $A$ being the {\it support} of its spectral measure $E_A(\cdot)$, $\C$ can be replaced with $\sigma(A)$ in the above definitions whenever appropriate \cite{Dunford1958,Dun-SchIII}.

In a complex Hilbert space, the scalar type spectral operators are precisely those similar to the {\it normal} ones \cite{Wermer1954}.

\section{Spectral Features Inherent to Scalar Type Spectral Operators}

In \cite{Markin2006}, the following generalization of the well-known orthogonal decomposition for a normal operator in a complex Hilbert space (see, e.g., \cite{Dun-SchII,Plesner}) is found:

\begin{thm}[{\cite[Theorem]{Markin2006}}]\label{M2006}\ \\
For a scalar type spectral operator $A$ in a complex Banach space $(X,\|\cdot\|)$ with spectral measure $E_A(\cdot)$, the \textit{direct sum decomposition}
\begin{equation}\label{DS}
X=\ker A\oplus \overline{R(A)}
\end{equation}
($\ker \cdot$ is the {\it kernel} of an operator) holds with 
\[
\ker A=E_A(\{0\})X\quad \text{and}\quad \overline{R(A)}=E_A(\sigma(A)\setminus\{0\})X.
\]
\end{thm}

Decomposition \eqref{DS} has the following immediate implication generalizing the well-known fact for \textit{normal operators} (see, e.g., \cite{Dun-SchII,Plesner}).  

\begin{cor}[Emptiness of Residual Spectrum]\label{ERS}\ \\
For a scalar type spectral operator $A$ in a complex Banach space $(X,\|\cdot\|)$,  $\sigma_r(A)=\emptyset$.
\end{cor}

\begin{proof}\quad
Whenever, for $\lambda\in\C$, the scalar type spectral operator $A-\lambda I$ is \textit{one-to-one}, $\ker(A-\lambda I)=\{0\}$, and hence, by \eqref{DS},
$\overline{R(A-\lambda I)}=X$, which implies that $\sigma_r(A)=\emptyset$.
\end{proof}

\begin{exmp}\label{exmp1}
In $l_2$, the unbounded linear operator 
\[
A(x_1,x_2,\dots)=(0,x_1,2x_2,\dots,nx_{n+1},\dots)
\] 
with the domain $D(A)=\left\{(x_1,x_2,\dots)\in l_2\,\middle|\, (0,x_1,2x_2,\dots,nx_{n+1},\dots)\in l_2\right\}$
is den\-sely defined and closed, but, by Corollary \ref{ERS}, is \textit{not spectral of scalar type} since
$0\in \sigma_r(A)$.
\end{exmp}

In respect that $\sigma_r(A)=\emptyset$, the proof of {\cite[Theorem 4.4]{Foguel1958}} can be used verbatim to prove the following

\begin{prop}[Countability of Point Spectrum]\label{CPS}\ \\
For a scalar type spectral operator $A$ in a complex separable Banach space $(X,\|\cdot\|)$, $\sigma_p(A)$ is a countable set.
\end{prop}

\begin{exmp}\label{exmp2}
In the separable Banach space $C([a,b],\C)$ ($-\infty<a<b<\infty$) with 
the \textit{maximum norm}, the 
differentiation operator
\[
C^1[a,b]\ni x\mapsto [Ax](t)=x'(t),
\ a\le t\le b,
\]
is densely defined, linear, and closed, but, by Proposition \ref{CPS}, \textit{not spectral of scalar type} since 
$\sigma_p(A)=\C$.
\end{exmp}

\smallskip
Now, let us stretch Theorem \ref{SGB} to the unbounded case.

\begin{thm}[Characterization of Spectral Gap at 0]\label{CSG}\ \\
For a scalar type spectral operator $A$ in a complex Banach space $(X,\|\cdot\|)$ with spectral measure $E_A(\cdot)$ and $0\in \sigma(A)$, $0$ is an isolated point of the spectrum $\sigma(A)$ iff the range $R(A)$ of $A$ is closed, i.e., $\overline{R(A)}=R(A)$.
\end{thm}

\begin{proof}\

{\it ``Only if"} part.\quad 
Suppose that $0$ is an \textit{isolated point} of $\sigma(A)$.

Considering that
\[
\sigma(A)\setminus \{0\}=\sigma(A)\setminus \left\{\lambda\in\C\,\middle|\, |\lambda|<\gamma\right\}.
\]
with some $\gamma>0$, to the \textit{bounded} Borel measurable function 
\begin{equation*}
F(\lambda):=  
\begin{cases}
0&\text{for $\lambda\in \C$ with $|\lambda|< \gamma$}\\
\dfrac{1}{\lambda}&\text{for $\lambda\in \C$ with $|\lambda|\ge \gamma$},
\end{cases}
\end{equation*}
by the properties of the {\it operational calculus} ({\cite[Theorem XVIII.2.11]{Dun-SchIII}}), there corresponds a \textit{bounded} scalar type spectral operator
\[
F(A)=\int\limits_{\C}F(\lambda)\,dE_A(\lambda)
\]
and, for each $f\in X$,
\begin{multline*}
E_A(\sigma(A)\setminus \{0\})f
=E_A\left(\sigma(A)\setminus \left\{\lambda\in\C\,\middle|\, |\lambda|<\gamma\right\}\right)f
\\
\shoveleft{
=\int\limits_{\left\{\lambda\in\C\,\middle|\, |\lambda|\ge \gamma\right\}} 1\,dE_A(\lambda)f
=\int\limits_{\C} \lambda F(\lambda)\,dE_A(\lambda)f
=AF(A)f\in R(A).
}\\
\end{multline*}

Since, by Theorem \ref{M2006}, $E_A(\sigma(A)\setminus \{0\})$ is the projection onto $\overline{R(A)}$ along $\ker A$ \cite{Markin2006}, we infer
that $\overline{R(A)}=R(A)$. 

\smallskip
{\it ``If"} part.\quad
Suppose that $\overline{R(A)}=R(A)$, which, considering $\sigma_r(A)=\emptyset$, implies that
$0\in \sigma_p(A)$, i.e., $\ker A\neq \{0\}$. 

Then, by Theorem \ref{M2006}, the direct sum decomposition
\begin{equation}\label{DSS}
X=\ker A\oplus R(A),
\end{equation}
where $\ker A=E_A(\{0\})X$ and $R(A)=E_A(\sigma(A)\setminus\{0\})X$, holds, and hence, $A$ can be treated as the matrix operator 
\[
\begin{bmatrix}
0&0\\
0&A_1\\
\end{bmatrix},
\]
in $\ker A\oplus R(A)$, where $A_1:D(A)\cap R(A)\to R(A)$ is the restriction of $A$ to $R(A)$. Such a consideration makes apparent the fact that 
\[
\sigma(A)=\{0\}\cup \sigma(A_1).
\]
Since $\ker A\cap R(A)=\{0\}$, the \textit{closed linear operator} $A_1:D(A)\cap R(A)\to R(A)$ is \textit{bijective} and has an \textit{inverse} defined on $R(A)$, which, in respect that
$(R(A),\|\cdot\|)$ is a Banach space, by the \textit{Closed Graph Theorem} (see, e.g., \cite{Dun-SchI}), is \textit{bounded}. 

Hence, $0$ is a \textit{regular point} of $A_1$. Considering the fact that the \textit{resolvent set} of a closed operator is \textit{open} in $\C$ (see, e.g., \cite{Dun-SchI}), we infer that, there is a neighborhood of $0$ not containing points of $\sigma(A_1)$, i.e., other points of $\sigma(A)$, which makes $0$ to be an \textit{isolated point} of $\sigma(A)$.
\end{proof}

\begin{rem}
Observe that, the fact that $\lambda_0$
is an \textit{isolated point} of the spectrum $\sigma(A)$ of a scalar type 
spectral operator $A$, necessarily implies that $\lambda_0\in \sigma_p(A)$. Indeed, the spectrum being the {\it support} of the operator's spectral measure $E_A(\cdot)$, we  immediately infer that
\begin{equation*}
E_A(\{\lambda_0\})\neq 0,
\end{equation*}
which makes $\lambda_0$ to be an \textit{eigenvalue} of $A$ with the \textit{eigenspace} $E_A(\{\lambda_0\})X$ \cite{Dunford1958,Dun-SchIII}. The converse, however, is not true.
\end{rem}

\begin{exmp}
In $l_2$, for the \textit{self-adjoint} operator 
\[
l_2\ni (x_1,x_2,\dots)\mapsto A(x_1,x_2,\dots)=(0,x_2,x_3/2,x_4/3,\dots)\in l_2,
\] 
the eigenvalue $0$ is not an isolated point of $\sigma(A)
=\sigma_p(A)=\left\{0,1,1/2,1/3,\dots\right\}$.
\end{exmp}

\begin{cor}\label{RI}
If, for a scalar type spectral operator $A$ in a complex Banach space $(X,\|\cdot\|)$, $0$ is a regular point or an isolated point of the spectrum $\sigma(A)$, direct sum decomposition \eqref{DSS} holds and the operator $A+E_A(\{0\})$ has a bounded inverse defined on $X$, i.e., $0\in \rho\left(A+E_A(\{0\})\right)$ ($\rho(\cdot)$ is the \textit{resolvent set} of an operator).
\end{cor}

\begin{proof}\quad 
The validity of decomposition \eqref{DSS}
follows immediately from Theorem \ref{CSG}.

Since, by Theorem \ref{M2006}, the projection $E_A(\{0\})$ is onto $\ker A$ along $R(A)$, the rest follows from a more general statement concerning the existence of a bounded inverse defined on $X$ of $A+P$ with a closed linear operator $A$, for which decomposition \eqref{DSS} holds, and $P$ is the projection onto $\ker A$ along $R(A)$ (\cite{Korolyuk-Turbin,Kato,Daletsky-Krein}, cf. also \cite{Markin1991,Markin1994Preprint,MarkinMEWS}). Such operators  are naturally called \textit{reducibly invertible}. 
\end{proof}

Thus, a scalar type spectral operator $A$, for which $0$ is a regular point or an isolated point of spectrum, is \textit{reducibly invertible}.

\section{Final Remarks}

As Examples \ref{exmp1} and \ref{exmp2} demonstrate, Corollary \ref{ERS} and Proposition \ref{CPS} are ready tests for disqualifying an operator from being scalar type spectral.

Theorem \ref{CSG} relates a peculiar topological property of the spectrum of a scalar type spectral operator to a rather natural topological property of its range.

Observe also that decompositions \eqref{DS} and \eqref{DSS} are essential in the context of the asymptotic behavior of \textit{weak/mild solutions} of the associated abstract evolution equation
\begin{equation*}
y'(t)=Ay(t),\ t\ge 0,
\end{equation*}
\cite{Ball,Engel-Nagel,Hille-Phillips,Markin1991,Markin1994Preprint,MarkinMEWS}.

\section{Acknowledgments}

The author extends cordial thanks to his colleague, Dr.~Michael Bishop of the Department of Mathematics, California State University, Fresno, for stimulating discussions which inspired writing this note.

\end{document}